\documentclass[12pt]{article}

\usepackage{graphicx}
\usepackage{latexsym,amssymb}
\usepackage{amsthm}
\usepackage{indentfirst}
\usepackage{amsmath}
\usepackage{color}
\usepackage{xcolor}
\usepackage{fourier}
\usepackage[colorlinks=true,backref=page]{hyperref}
\usepackage{enumerate}

\newtheorem{theoremalph}{Theorem}

\newtheorem{corollary-main}[theoremalph]{Corollary}
\newtheorem{Theorem}{Theorem}[section]
\newtheorem*{Theorem A}{Theorem A}
\newtheorem*{Theorem A'}{Theorem A'}
\newtheorem*{Thm}{Theorem}

\newtheorem*{Conj*}{Conjecture}

\newtheorem{Proposition}[Theorem]{Proposition}
\newtheorem{Lemma}[Theorem]{Lemma}

\newtheorem{Remark-numbered}[Theorem]{Remark}
\newtheorem{Remarks-numbered}[Theorem]{Remarks}
\newtheorem{Corollary}[Theorem]{Corollary}

\newtheorem*{Claim}{Claim}
\newtheorem*{Fact}{Fact}
\newtheorem{Claim-numbered}{Claim}

\def\MM{{\mathbb M}} \def\NN{{\mathbb N}}

\def\cA{{\cal A}}   \def\cM{{\cal M}} 
\def\cB{{\cal B}}    
\def\cC{{\cal C}}    
    \def\cV{{\cal V}}

\def\dim{\operatorname{dim}}

\def\Leb{\operatorname{Leb}}

\begin{document}

\title{SRB measures for  mostly expanding partially hyperbolic diffeomorphisms via the variational approach}

\author{David Burguet and Dawei Yang\footnote{Just after submitting the first version to Arxiv, the second author learned that the first author had also obtained similar results, which were notably presented at the Partial Hyperbolicity Summer School in Maryland (May 2023). The authors therefore decided to do this work together. 
D. Yang  was partially supported by National Key R\&D Program of China (2022YFA1005801), by NSFC 12171348 and NSFC 12325106.
}}


\maketitle
\begin{abstract}
By using the variational approach, we prove  the existence of Sinai-Ruelle-Bowen measures for partially hyperbolic $\mathcal C^1$ diffeomorphisms with  mostly expanding properties. The same conclusion holds true if one considers a dominated splitting $E\oplus F$, where $\dim E=1$ and $F$ is mostly expanding.  When the diffeomorphisms are $\mathcal C^{1+\alpha}$, we prove the basin covering property for both cases.
\end{abstract}
\section{Introduction}

In differentiable ergodic theory, Sinai-Ruelle-Bowen measures are important objects that have nice variational, geometrical and observable properties. The SRB theory has been established for uniformly hyperbolic systems by a sequence of works by Sinai, Ruelle and Bowen in the last seventies \cite{Sin72,BoR75,Rue76}.

\smallskip

Beyond uniform hyperbolicity, an important notion ``partial hyperbolicity'' studied by Brin, Pesin and Sinai \cite{PeS82} attracted people's attention from the eighties.

Let $M$ be compact $d$-dimensional Riemannian manifold and $f:~M\to M$ be a $\mathcal C^1$ diffeomorphism. For a compact invariant set $\Lambda$, a $Df$-invariant sub-bundle $E\subset T_\Lambda M$ is said to be \emph{contracted}, if there are $C>0$ and $\lambda\in(0,1)$ such that for any $x\in\Lambda$ and any $n\in\NN$, one has $\|Df^n|_{E(x)}\|\le C\lambda^n$. An invariant bundle $E$ is said to be \emph{expanded} if it is contracted for $f^{-1}$. For two $Df$-invariant sub-bundles $E,F\subset T_\Lambda M$, one says that $E$ is \emph{dominated} by $F$, if there are $C>0$ and $\lambda\in(0,1)$ such that for any $x\in\Lambda$ and any $n\in\NN$, one has $\|Df^n|_{E(x)}\|\|Df^{-n}|_{F(f^nx)}\|\le C\lambda^n$.

We will mainly consider the following \emph{partially hyperbolic} splitting $T_\Lambda M=E\oplus F$, where $E$ is contracted and $E$ is dominated by $F$. It was asked generally how the SRB theory can be established for general dynamical systems, see for instance  \cite{Pal05}. Since partially hyperbolic diffeomorphisms form an open set in the space of diffeomorphisms, an immediate task is to establish the SRB theory in the partially hyperbolic setting.

One can define SRB measures from two different points of view: the variational property and the geometric property. For a $\mathcal C^1$ diffeomorphism $f$, given an invariant measure $\mu$,   we let 
$$\lambda_d(x)\le  \cdots\le  \lambda_2(x)\le \lambda_1(x)$$
 be the  Lyapunov exponents given by  Oseledets theroem for $\mu$-almost every point $x$.   By Ruelle inequality \cite{Rue78},   the entropy $h_\mu(f)$ of $\mu$ is always bounded from above by $\int\sum_{\lambda_i(x)>0}\lambda_i(x){\rm d}\mu(x)$.  The measure $\mu$ is said to be a \emph{Sinai-Ruelle-Bowen measure}, or an SRB measure for short,  when $\lambda_1(x)>0$ for $\mu$ a.e.  $x$ and the equality holds, i.e.  
\begin{equation}\label{Pesinformula}h_\mu(f)=\int\sum_{\lambda_i(x)>0}\lambda_i(x){\rm d}\mu(x)>0,\end{equation}
 This is the \textit{variational property} of SRB measures. When $f$ is a $\mathcal C^{1+\alpha}$ diffeomorphism, and if $\mu$ has positive Lyapunov exponents, then by Pesin theory, one knows the existence of unstable manifolds for $\mu$-almost every point and one may construct   measurable partitions subordinate to these unstable manifolds. If the conditional measures of $\mu$ along such a  measurable partition are absolutely continuous with respect to the Lebesgue measures of unstable manifolds, then $\mu$ is also said to be an SRB measure. These two definitions are compatible because  they are equivalent for $\mathcal C^{1+\alpha}$ diffeomorphisms according to a well-know theorem of Ledrappier and Young \cite{LeY85} (see also \cite{Led84, Bro22}).  In this context, it follows from Pesin theory, that an ergodic SRB measure $\mu$,  which is hyperbolic, i.e.  $\forall i, \int \lambda_i\,d\mu\neq 0$,  has a basin of positive Lebesgue measure.

There are many methods to build SRB measures  : by coding with Markov partitions, by pushing the Lebesgue measure on unstable disk (the so-called geometrical approach), by inducing with   Gibbs-Markov-Young structure,  by taking zero-noise limits of random perturbations, ...

We recall now a pionneer  work on SRB measures for partially hyperbolic systems.  
A compact set $\Lambda$ of $M$  is said to be an \emph{attractor} of a $\mathcal C^1$ diffeomorphism $f:M\rightarrow M$ if there is a neighborhood $U$ of $\Lambda$ such that $f(\overline{U})\subset U$ and $\bigcap_{n\in\NN}f^n(U)=\Lambda$.  Recall  also that when $L:~V_1\to V_2$ is a linear isomorphism between two normed linear spaces $V_1$ and $V_2$, the mini-norm or the co-norm of $L$ is $m(L)$ and is defined by the following way:
$$m(L)=\inf_{v\in V_1\setminus\{0\}}\frac{\|Lv\|}{\|v\|}.$$ A  continuous bundle $F\subset T_\Lambda M$,  may be extended continuously on a neighborhood $U$ of $\Lambda$. 
Then one may define the upper assympotic Birkhoff sum of $m(Df|_F)$ as follows:
$$\forall x\in U, ~~~\ \overline{m}_F(x,f):=  \limsup_{n\to+\infty}\frac{1}{n}\sum_{i=0}^{n-1}\log m(Df|_{F(f^ix)}).$$
This quantity does not depend on the choice of the extension of $F$. 

\begin{Thm}[\cite{ABV00,ADLP}] Assume that $\Lambda=\bigcap_{n\in \NN}f^nU$ is an attractor  of  a $\mathcal C^{1+\alpha}$ diffeomorphism $f$ with a dominated splitting $T_\Lambda M=E\oplus  F$ with $E$ being uniformly contracting.  

Then Lebesgue a.e.  $x\in U$ with  \begin{equation}\label{mostexp}
 \overline{m}_F(x,f)>0
  \end{equation}
lies in the basin of an ergodic hyperbolic SRB measure.  

Moreover for any $a>0$, the set $\left\{x\in U, \  \overline{m}_F(x,f)>a\right\}$ may be  covered Lebesgue almost everywhere by the basins of finitely many ergodic hyperbolic SRB measures.  
\end{Thm}

When the set of points $x$ satisfying (\ref{mostexp}) has positive Lebesgue measure, the partially hyperbolic system $f$ is called mostly expanding. 
The above theorem was proved by Alves, Bonatti and Viana in \cite{ABV00} with the geometrical approach, when the ``limsup'' is replaced by the ``liminf'' .  The  present statement is due to Alves,  Dias, Pinheiro and Luzzato \cite{ADLP}.  They assume that  the diffeomorphism is  $\mathcal  C^{1+\alpha}$ and the contracting bundle is necessary because they used the geometric properties of the SRB measures and  \emph{Gibbs-Markov-Young} structure.

In many  recent works  the variational property was used  to construct SRB measure \cite{BCS23,CCE15,CYZ20,CaY16,LeY17,Qiu11, Bur21}.  In this paper, we will show the existence of SRB measures for $\mathcal C^1$ mostly expanding partially hyperbolic diffeomorphisms by using  entropic methods. This avoids the construction of the Gibbs-Markov-Young structure and proves the existence of SRB measures for a larger class of diffeomorphisms.

%

Before stating our main results we introduce the lower empirical exponent $\chi^F_{\mathrm{min}}$ inside $F$ as follows : $$\forall x\in U,\ \ \ \chi^F_{\mathrm{min}}(x,f):= \lim_{p\to\infty}\frac{1}{p}\limsup_{n\to+\infty}\frac{1}{n}\sum_{i=0}^{n-1}\log m\left(Df^p|_{F(f^ix)}\right).$$

Contrarily to the quantity  $\overline{m}_F(x,f)$, the exponent $\chi^F_{\mathrm{min}}$ does not depend on the choice of the Riemannian structure and it is homogenous in $f$, i.e. $\forall p\in \NN, \ \ \chi^F_{\mathrm{min}}(x,f^p)=p\chi^F_{\mathrm{min}}(x,f)$.  When there is no confusion on the map $f$,  we just let  $\chi^F_{\mathrm{min}}(x)$ for $ \chi^F_{\mathrm{min}}(x,f)$.  
We always have  $\chi^F_{\mathrm{min}}(x,f)\leq \limsup_n \frac{1}{n}\log m(Df^n|_{F}(x))$, but this last  inequality is  a priori strict for a general  point $x$.  When $\mu$ is an ergodic measure,  $\mu$ a.e.  point $x$ satisfies  $$\chi^F_{\mathrm{min}}(x,f)=\lim_n \frac{1}{n}\log m(Df^n|_{F(x)})=\lambda_{\mathrm{dim}(F)}(x)=\int \lambda_{\mathrm{dim}(F)}\, d\mu.$$ Moreover it follows from the definitions that $\overline{m}_F(x,f)\leq \chi^F_{\mathrm{min}}(x) $.   The equality does not hold in general, even for typical points $x$ with respect to invariant measures. 
Let $pw(x)$ be the set of empirical measures from $x$, that is the set of limits for the weak-$*$ topology, when $n$ goes to infinity,  of the sequence $\mu_x^n:=\frac{1}{n}\sum_{0\leq k<n }\delta_{f^kx}$, $n\in \NN$.  The exponent $\chi^F_{\mathrm{min}}(x,f)$ is related with  the Lyapunov exponent of empirical measures as follows (see  the appendix):
\begin{equation}\label{eq:empexp}\chi^F_{\mathrm{min}}(x,f) =\sup_{\mu\in pw(x)}\int \lambda_{\mathrm{dim}(F)}\, d\mu.
\end{equation}
For all these reasons we prefer to state our main results by using  $\chi^F_{\mathrm{min}}$,  but this is mostly a question of presentation.  Indeed one sees easily 
$$\left\{\chi^F_{\mathrm{min}}(\cdot,f)>0\right\}=\bigcup_{p\in \NN^*}\bigcup_{0=k<p}f^{-k}\left\{\overline{m}_{F}(\cdot,f^p)>0\right\}.$$  Therefore by applying the above theorem to the powers of $f$, we get in the same settings the apparently stronger result,  that Lebesgue a.e.  $x\in U$ with  $\chi^F_{\mathrm{min}}(x,f)>0$ 
lies in the basin of an ergodic hyperbolic SRB measure.  When there is no confusion on $f$, we just  write $\overline{m}_F(x)$ and $\chi^F_{\mathrm{min}}(x)$ for $\overline{m}_F(x,f)$ and $\chi^F_{\mathrm{min}}(x,f)$ respectively.

\begin{theoremalph}\label{Thm:dominated-entropy-formula}
Assume that $\Lambda=\bigcap_{n\in \NN}f^nU$ is an attractor  of  a $\mathcal C^{1}$ diffeomorphism $f$ with a dominated splitting $T_\Lambda M=E\oplus  F$.  

If 
$${\rm Leb}\left(\left\{x\in U:~\chi^{F}_{\mathrm{min}}(x)>a\right\}\right)>0 \text{ for some }a\geq 0,$$
then there is an ergodic  measure $\mu$ supported on $\Lambda$ such that

 \begin{equation*}\int \lambda_{\mathrm{dim}(F)}\, d\mu>a
\end{equation*}
and  \begin{equation*}h_\mu(f)\ge \int\log{\rm Jac}(Df|_{F}) d\mu.
 \end{equation*}

\end{theoremalph}

If the unstable index $i_u(\mu):=\sharp\left\{i, \ \int \lambda_i\, d\mu>0\right\}$ of $\mu$ satisfies moreover $i_u(\mu)\leq \mathrm{dim}(F)$,  then $\mu$  is an SRB measure.   In particular, this is the case when  $a\geq \underline{a}$ with 
$$\underline{a}:=\sup\left\{\int \lambda_{\mathrm{dim}(F)}\, d\mu \ : \ \mu \in \underline{\mathcal M}\right\}$$
and 
$$\underline{\mathcal M}:=\left\{\mu  \text{ ergodic:  }h_\mu(f)\ge \int\log{\rm Jac}(Df|_{F}){\rm d}\mu \text{ and } i_u(\mu)>  \mathrm{dim}(F)\right\}.$$

\begin{corollary-main}\label{corCone}Assume that $\Lambda=\bigcap_{n\in \NN}f^nU$ is an attractor  of  a $\mathcal C^{1}$ diffeomorphism $f$ with a dominated splitting $T_\Lambda M=E\oplus F$.

  If 
$${\rm Leb}\left(\left\{x\in U :~\chi^{F}_{\mathrm{min}}(x)>\max(\underline{a},0)\right\}\right)>0,$$
 then $f$ has an ergodic  SRB measure $\mu$ supported on $\Lambda$ with $\int \lambda_{\mathrm{dim}(F)}\, d\mu >\max(\underline{a},0)$.
 \end{corollary-main}

Under the conditions of Corollary~\ref{corCone}, if the center-unstable index  $i_{cu}(\mu):=\sharp\left\{i, \ \int \lambda_i\, d\mu\geq 0\right\}$ of $\mu$ satisfies $i_{cu}(\mu)\leq \mathrm{dim}(F)$,  the measure  $\mu$ is an  ergodic  hyperbolic SRB measure. When $f$ is $\mathcal C^{1+\alpha}$, its basin has then positive Lebesgue measure.   We let  now 
$$(\underline{a}\leq ) \ \overline{a}:=\sup\left\{\int \lambda_{\mathrm{dim}(F)}\, d\mu \ : \ \mu\in \overline{\mathcal M}\right\}$$
and 
$$\overline{\mathcal M}:=\left\{ \mu \text{ ergodic:  }h_\mu(f)\ge \int\log{\rm Jac}(Df|_{F})d\mu \text{ and  } i_{cu}(\mu)>  \mathrm{dim}(F)\right\}. $$ 
 Such a \emph{physical} measure  $\mu$  exists whenever the set $\left\{x\in U :~\chi^{F}_{\mathrm{min}}(x)>\max(\overline{a},0)\right\}$ has positive Lebesgue measure.  Moreover, the basins of these measures cover this set Lebesgue almost everywhere:
 \begin{corollary-main}\label{corConeplus}
Assume that $\Lambda=\bigcap_{n\in \NN}f^nU$ is an attractor  of  a $\mathcal C^{1+\alpha}$ diffeomorphism $f$ with a dominated splitting $T_\Lambda M=E\oplus F$.  

Then Lebesgue a.e.   $x\in U$  with 
$\chi^{F}_{\mathrm{min}}(x)>\max(\overline{a},0)$ lies in the basin of an ergodic hyperbolic SRB measure $\mu$ supported on $\Lambda$ with $\int \lambda_{\mathrm{dim}(F)}\, d\mu >\max(\overline{a},0)$.

 \end{corollary-main}

The following properties implies various versions of the above corollaries under stronger assumptions: 
\begin{enumerate}[i.]
\item if $E$ is uniformly contracting, then $\overline{\mathcal M}=\emptyset$  (thus $\overline{a}=-\infty$),
\item if $\lim_n\sup_{x\in \Lambda}\frac{1}{n}\log\| D_xf^n|_{E}\|\leq 0$, then $\underline{\mathcal M}=\emptyset$ (thus $\underline{a}=-\infty$), 
\item if $\mathrm{dim}(E)=1$, then $\overline{a}\leq 0$,  
\item ($\Rightarrow $iii.) \ $\overline{a}\leq \frac{\Lambda_{\mathrm{dim}(E)-1}(f^{-1})}{\mathrm{dim}(F)}$.
\end{enumerate}
The two first items follow straightforwardly from the definitions. Let us justify the last item. 
Recall that $\Lambda_k(f):=\max_{l\leq k}\lim_{n\rightarrow +\infty}\sup_{x\in M}\frac{1}{n}\log \|\Lambda^{l}Df^n(x)\|$, where  $\Lambda^{l}Df$ denotes the map induced by $f$ on the $l$-exterior power of the tangent bundle.  By Ruelle inequality \cite{Rue78}, any ergodic measure $\mu$ with $i_{cu}(\mu)>\mathrm{dim}(F)$ satisfies $h_\mu(f) =h_{\mu}(f^{-1})\leq \Lambda_{\mathrm{dim}(E)-1}(f^{-1})$.  If $\mu$ satisfies moreover $ h_\mu(f)\ge \int\log{\rm Jac}(Df|_{F}){\rm d}\mu$, then we have $\int \lambda_{\mathrm{dim}(F)}\, d\mu \leq \frac{h_\mu(f)}{\mathrm{dim}(F)}\leq  \frac{\Lambda_{\mathrm{dim}(E)-1}(f^{-1})}{\mathrm{dim}(F)}$.  \\

In the first context (i),  Corollary \ref{corConeplus} recovers the aforementionned result  of Alves,  Dias,   Pinheiro and Luzzato \cite{ADLP}.    Existence of SRB measures for dominated splittings satisfying ${\rm Leb}\left(\left\{x:~\chi^{F}_{min}(x)>0\right\}\right)>0$ in the last two contexts (ii) and (iii),  where existence of  Gibbs-Markov-Young structure is not known,  is new. \\

We also investigate the finiteness of ergodic SRB measures in the $\mathcal C^{1+\alpha}$ setting.   

\begin{theoremalph}\label{Thm:ergodic-finiteness-away-from-zero}
Assume that $\Lambda=\bigcap_{n\in \NN}f^nU$ is an attractor  of  a $\mathcal C^{1+\alpha}$ diffeomorphism $f$ with a dominated splitting $T_\Lambda M=E\oplus F$.  
 For any $a>0$,  there are only finitely many ergodic SRB supported on $\Lambda$ satisfying 
$$\int \log m(Df|_F) \, d\mu >a \text{ and }\int \log \|Df|_E\| \, d\mu <-a.$$
\end{theoremalph}

The finiteness property in the aforementioned theorem of Alves, Dias, Pinheiro and Luzzato 
follows then from Theorem \ref{Thm:ergodic-finiteness-away-from-zero}.  When ${\rm dim}(E)=1$ with $E$ not necessarily contracting,  we can also conclude that for any $a$ the set $\{x\in U, \ \overline{m}_F(x)>a\}$ is covered by the basin of finitely many ergodic SRB measures. Indeed if $x\in U$ with  $\overline{m}_F(x)>a$ lies in the basin of an ergodic SRB measure $\mu$ we have $h_\mu(f)\geq (\dim F) \int \log m(Df|_F)\, d\mu>(\dim F) a$ and  $\int \lambda_{d}\, d\mu=\int \log \|Df|_E\| \, d\mu$ as ${\rm dim}(E)=1$ so that Ruelle inequality yields 
$$(\dim F)a<h_{\mu}(f^{-1})\leq -\int \log \|Df|_E\| \, d\mu.$$  However we do not expect the finiteness property if one works with $\chi_{{\rm min}}^F$ rather than $\overline{m}_F$. \\

The above results  also holds true for $\mathcal C^1$-local diffeomorphisms.   A $\mathcal C^1$ map $f:~M\to M$ is said to be a \emph{$\mathcal C^1$-local diffeomorphism} if for any $x\in M$, there is a neighborhood $V$ of $x$ such that $f$ is a diffeomorphism from $V$ to its image.  In this context,  an invariant measure $\mu$ is said to be an \textit{expanding  SRB measure} when its Lyapunov exponents are all positive and $\mu$ satisfies the entropy formula (\ref{Pesinformula}). We refer to \cite{QiZ02} for   a version of Ledrappier-Young's theorem for $\mathcal C^{2}$ local diffeomorphisms. 
For a local diffeomorphism, one defines
$$\chi_{{\rm min}}(x):=\lim_p\frac{1}{p}\limsup_{n\to\infty}\frac{1}{n}\sum_{i=0}^{n-1}\log m(Df^p(f^ix))\ \left(=\sup_{\mu\in pw(x)} \int \lambda_d \, d\mu\right).$$

\begin{theoremalph}\label{Thm:local}
Assume that $f$ is a $\mathcal C^1$-local diffeomorphism. 
If 
$${\rm Leb}\left(\left\{x\in U:~\chi_{{\rm min}}(x)>a\right\}\right)>0 \text{ for some }a\geq 0,$$
 then $f$ has an ergodic SRB measure with $\int \lambda_i d\mu>a$ for any $i$. \\
 
 Moreover, if $f$ is $\mathcal C^{2}$, 
then Lebesgue a.e.   $x\in M$  with 
$\chi_{{\rm min}}(x)>0$ lies in the basin of an expanding  SRB measure. 
\end{theoremalph}

\subsubsection*{Acknowledgements.}
We are grateful to J. Chen for his help in ergodic theory, for J. Zhang and Z. Mi for their help to improve the presentation.

\section{Construction of SRB measure}
 In \cite{Bur21} the first author introduced an entropic variation of the so-called geometric approach for building SRB measures for $C^r$, $r>1$,  smooth surface diffeomorphisms.  We follow here the same strategy.   A key property to construct SRB measures is the \textit{bounded distorsion  property}, which follows here from the domination property whereas it is ensured by the smoothness  in \cite{Bur21}.

The geometric approach for uniformly hyperbolic systems consists in considering a weak limit of $\left(\frac{1}{n}\sum_{k=0}^{n-1}f_*^k{\rm Leb}_{D_u}\right)_n$, where $D_u$ is a local unstable disk and  ${\rm Leb}_{D_u}$ denotes the normalized Lebesgue measure on $D_u$ induced by its inherited Riemannian structure as a submanifold of $M$.

In this section we prove Theorem A.  Without loss of generality (consider some power of $f$) we may assume that the set $\left\{(x\in U, \ \overline{m}_F(x)>a\right\}$ with $a\geq 0$  has positive Lebesgue measure.  Then, by a standard Fubini argument,  we may consider a smooth disk $D$ tangent to a small cone around the bundle $F$ such that 
$\mathrm{Leb}_{D}\left(\left\{(x\in U, \ \overline{m}_F(x)>a\right\}\right)>0$.  In \cite{ABV00}
 the authors show by using a bounded geometry property (see Lemma \ref{Lem:Pliss-iterate} below)  at Pliss times for $\left(\log m\left(Df|_{F(f^kx)}\right)\right)_k$ that any weak-$*$ limit of 
$\left(\frac{1}{n}\sum_{k=0}^{n-1}f_*^k{\rm Leb}_{D}\right)_n$ admits some ergodic SRB component.  In fact they need to assume  ${\rm Leb}_{D}\left(\left\{(x\in U, \ \underline{m}_F(x)>0\right\}\right)>0$ with $$ \underline{m}_F(x):=\liminf_{n\to+\infty}\frac{1}{n}\sum_{i=0}^{n-1}\log m(Df|_{F(f^ix)}).$$

We build  SRB measures as weak-$*$ limit of   F\o lner empirical measures,  i.e.  limit $\mu$ of a sequence of the form  $\left(\frac{1}{\sharp Q_\ell}\sum_{k\in Q_\ell}f_*^k\mu_{\ell}\right)_\ell$ with $n_\ell\xrightarrow{\ell}+\infty$ such that:
\begin{itemize}
\item $(Q_\ell)_\ell$ with $Q_\ell\subset [0,n_\ell)$ is a F\o lner sequence,  so that the weak limit $\mu$ will be invariant by $F$,
\item for all $\ell$, the measure $\mu_{\ell}=\frac{\Leb_{D}(\cdot)}{\Leb_{D}(\Lambda_\ell)}$ is the probability measure induced by $\Leb_{D}$ on $\Lambda_\ell\subset \{\overline{m}_F>a\} $,  the $\Leb_{D}$-measure of $\Lambda_\ell$ being  not exponentially small in $n_\ell$. 
\item the sets $(Q_\ell)_\ell$ are in some sense  \textit{filled with } the  set  $P(x)$ of Pliss times for $x\in \Lambda_\ell$. Then the measure $\mu$ has only positive Lyapunov exponent inside  $F$.

\end{itemize} 
Somehow,  this set $\Lambda_\ell$ (which follows from a Borel-Cantelli argument) is the price to pay to deal with $\overline{m}_F$ rather than $\underline{m}_F$.   As it has not exponentially small measure,  it will not affect our lower estimate on the entropy of $\mu$.  Finally we check  some F\o lner Gibbs property for $\mu_n$ : for any $\varepsilon>0$ we have for any partition $\mathcal A$ with small enough diameter and for all $\ell$ large enough: 
$$ \forall x\in \Lambda_\ell,  \ \ \mathrm{Leb}_{D}\left(\mathcal A^{Q_\ell}(x)
\cap \Lambda_{\ell} \right) \leq e^{\varepsilon n_\ell} \prod_{k\in Q_\ell }\mathrm{Jac}\left(Df|
_{F(f^kx)}\right),$$
where $\mathcal A^{Q_\ell}(x)$ denotes the atom of the iterated partition $\mathcal A^{Q_\ell}:=\bigvee_{k\in Q_\ell}f^{-k}\mathcal A$.   Then by a classical entropy computation, we conclude that $\mu$ satisfies $h_\mu(f)\ge \int\log{\rm Jac}(Df|_{F}){\rm d}\mu$.

\subsection{Submanifolds tangent to a cone around $F$}

Without loss of generality one may assume that $F$ can be extended continuously on $U$  and that $E$ can be extended continuously and invariantly on $U$.  For any point $x\in U$, for any $\theta>0$,  we define the cone $\cC^{F}_\theta(x)$ as follows:
$$\cC^{F}_\theta(x)=\left\{v=v^{E}+v^{F}\in T_x M:~\|v^{E}\|<\theta\|v^{F}\|\right\}.$$ 
By changing the metric if necessary,
 one can assume that for any $x\in U$,
$$Df(\cC^{F}_\theta(x))\subset \cC^{F}_{\theta/2}(fx).$$
A submanifold $\Delta\subset U$ (maybe with boundary)   is said to be tangent to $\cC^{F}_\theta$ if for any $x\in D$, one has that $T_x \Delta\subset \cC^{F}_\theta(x)$.  If $\Delta$ is tangent to $\cC^{F}_\theta$, then $f(\Delta)$ is tangent to $\cC^{F}_{\theta/2}$.

For a submanifold $\Delta$ we let ${\rm Leb}_\Delta$ and $d_\Delta$ be respectively  the Lebesgue measure and the Riemannian distance   on $\Delta$  inherited from the induced Riemannian structure on $\Delta$.

\begin{Lemma}\label{Lem:estimate-at-bowen-ball}
For any $\varepsilon>0$, there are $\delta_\varepsilon>0$ and $\theta_\varepsilon>0$ such that for any submanifold $\Delta$ tangent to $ \cC^{F}_{\theta_\varepsilon}$ of dimension ${\rm dim}(F)$,  for any $x\in\Delta$, for any $n\in\NN$ and for any subset $\Gamma\subset\left\{y\in\Delta:~d(f^iy,f^ix)<\delta_\varepsilon,~0\le i\le n-1\right\}$,  we have:
$${\rm e}^{-n\varepsilon}{\rm Leb}_{f^n(\Delta)}(f^n(\Gamma))\le {\rm Leb}_\Delta(\Gamma)\prod_{i=0}^{n-1}{\rm Jac}\left(Df|_{F(f^ix)}\right)\le {\rm e}^{n\varepsilon}{\rm Leb}_{f^n(\Delta)}(f^n(\Gamma)).$$
\end{Lemma}
\begin{proof}

Given $\varepsilon>0$, there is $\theta_\varepsilon>0$, such that for any submanifold $\tilde \Delta$ tangent to $ \cC^{F}_{\theta_\varepsilon}$ and for any $z\in \tilde \Delta$,  we have:
\begin{equation}\label{e.jac-submanifold-bundle}
{\rm e}^{-\varepsilon/2}\le \frac{{\rm Jac}\left(Df|_{T_z \tilde \Delta}\right)}{{\rm Jac}\left(Df|_{F(z)}\right)}\le {\rm e}^{\varepsilon/2}.
\end{equation}
Then we may choose $\delta_\varepsilon>0$ such that  any $z,z'\in U$ with $d(z,z')<\delta_\varepsilon$ satisfy:
\begin{equation}\label{e.jac-submanifold}
{\rm e}^{-\varepsilon/2}\le \frac{{\rm Jac}\left(Df|_{F(z)} \right)}{{\rm Jac}\left(Df|_{F(z')}\right)}\le {\rm e}^{\varepsilon/2}.\end{equation}

By change of variables,  we get therefore for any $x\in U$ and for any  subset $\Gamma$ \\of  $\left\{y\in\Delta:~d(f^i(y),f^ix)<\delta_\varepsilon,~0\le i\le n-1\right\}$: 
\begin{align*}
{\rm Leb}_\Delta(\Gamma)&=\int_{f^n(\Gamma)}{\rm Jac}(Df^{-n}|_{T_{z}f^n(\Delta)}){d}{\rm Leb}_{f^n(\Delta)}(z)\\
&=\int_{f^n(\Gamma)}\prod_{i=0}^{n-1}\frac{1}{{\rm Jac}\left(Df|_{T_{f^{i}(f^{-n}z)}f^{i}(\Delta)}\right)}{d}{\rm Leb}_{f^n(\Delta)}(z).
\end{align*}
Since $f^{i}(f^{-n}z)$ is $\delta_\varepsilon$-close to $f^ix$ and $f^i\Delta$ is tangent to $\cC^{F}_{\theta_{\varepsilon}}$,  we conclude  with \eqref{e.jac-submanifold} and \eqref{e.jac-submanifold-bundle} that: 
\begin{align*}
{\rm e}^{-n\varepsilon}{\rm Leb}_{f^n(\Delta)}\left(f^n(\Gamma)\right)\le{\rm Leb}_\Delta(\Gamma)\prod_{i=0}^{n-1}{\rm Jac}\left(Df|_{F(f^ix)}\right)\le{\rm e}^{n\varepsilon}{\rm Leb}_{f^n(\Delta)}\left(f^n(\Gamma)\right).
\end{align*}
\end{proof}

\subsection{Pliss times}\label{Subsection:Pliss}
Let $a''>a$ with ${\rm Leb}\left(\left\{(x\in U, \ \overline{m}_F(x)>a''\right\}\right)>0$ and fix $a'\in ]a,a''[$. 
 An integer $n\in\NN$ is said to be a Pliss time at $x\in U$  if for any $0\le j\le n-1$,  one has 
$$\frac{1}{n-j}\sum_{i=0}^{n-j-1}\log m\left(Df|_{F(f^{j+i}x)    }\right)\ge a'.$$
This notion of Pliss times depends on the fixed parameter $a'$.  
By convention,  $n=0$ is a Pliss time. 
We denote by $P(x)\subset\NN$ the set of Pliss times at $x$.  Clearly the set of Pliss time is measurable,  i.e.  for all $n\in \NN$,  the set $\{x\in M, \ n\in P(x)\}$ is measurable.   For any subset $T\subset \NN$ and for any $n\in\NN$,  we let 
$$d_n(T)=\frac{\#(T\cap [0,n))}{n}.$$
 The upper density $ {\overline d}(T)$ of $T$ is then defined as 
$${\overline d}(T)=\limsup_{n\to\infty}d_n(T).$$

We state the following version of Pliss Lemma \cite{Pli72}  :
\begin{Lemma}[Lemma 3.1 in \cite{ABV00}]\label{Lem:Pliss}
There is a constant $\alpha=\alpha(a'',a')\in(0,1]$ such that  any $x$ with $ \overline{m}_F(x)>a''$ satisfies  ${\overline d}(P(x))\ge\alpha$.

\end{Lemma}

The dynamic on a disk $\Delta_x\ni x$ tangent to a cone around $F$    is backward contracting at  Pliss times $n\in P(x)$.  Moreover the geometry of $f^n\Delta_x$ is bounded at $f^nx$,  i.e.    it contains a disc  $\Delta_{f^nx}\ni f^nx$ with lower bounded size.
An embedded disc  $D$ is said centered at $x$ with radius $\delta$ when $d_D(x,y)=\delta$  for any $y$ in the boundary of $D$.

\begin{Lemma}[Lemma 4.2 in \cite{AlP08}]\label{Lem:Pliss-iterate} 
There is $\delta_0>0$ and $\theta_0>0$  and $N\in \NN$  such that for any disk $\Delta\subset U$   of radius $\delta_0$ tangent to $\cC^{F}_{\theta_0}$,  for any $x\in \Delta$ with $d_\Delta(x, \partial \Delta)\geq \frac{\delta_0}{2}$,  for any Pliss time $P(x)\ni n>N$,  the image $f^n(\Delta)$ contains a disk $\Delta_{f^nx}$ centered at $f^n x$ with radius $\delta_0$  such that $f^{-i}(\Delta_{f^n(x)})$ decays exponentially:  \begin{equation}\label{contractPl}\forall 0\le i\le n \  \forall y,z\in \Delta_{f^nx}, \ \ d_{f^{-i}(\Delta_{f^nx})}\left(f^{-i}x,f^{-i}y\right)\le e^{-ia}d_{\Delta_{f^nx}}(y,z).\end{equation}
Moreover when $f$ is $\mathcal C^{1+\alpha}$ there is a constant $C>1$ such that 
\begin{equation}\label{equation:distorsion}\forall y,z\in \Delta_{f^nx},  \ \frac{\mathrm{Jac}f^{-n}|_{T_y\Delta_{f^nx}} }{\mathrm{Jac}f^{-n}|_{T_z\Delta_{f^nx}} }<C.
\end{equation}
\end{Lemma}

For two points $x, y$ with $d(x,y)$ less than the radius of injectivity $R_{inj}$ of $M$,  we write $x\sim^F y$ when the unique geodesic joining $x$ and $y$ is tangent to  $\cC^{F}_{\theta_0}$.  Then we may find $\delta_1<\min(R_{inj},\delta_0)$ (depending on $\theta_0$) so small that if $x\sim^F y$,  $d(x,y)<\delta_1$  and $ d(fx,fy)<\delta_1$ then $f(x)\sim^Ff(y)$. 
For general disks  tangent to $\cC^{F}_{\theta_0}$  it may happen that $x\nsim ^F y$ for  some points $x$ and $y$ in the disk.   We will work with specific disks,  that may be written through the exponential map as the graph of a Lipschitz map over $F$.   For $\delta_1$ small enough,   any disk centered at $x$ tangent to $\cC^{F}_{\theta_1}$ for any  $\theta_1<\theta_0$ of radius $\delta_0$ contains a subdisk of the form ${\rm exp}_x(\Gamma_g)$ where $\Gamma_g$ is the graph  of a $\theta_1$-Lipshitz map $g:F(x)^{\delta_1}\rightarrow E(x)$ with $g(0)=0$ where     $F(x)^{\delta_1} $ denotes the ball centered at $0$ of radius $\delta_1$ in $F(x)\subset T_xM$.  Fix finally $\theta_1<\theta_0$  and $\delta_1$ so small that any two points $y,z$ in ${\rm exp}_x(\Gamma_g)$ satisfy $y\sim^Fz$.  
The parameters $\theta_0$,  $\theta_1, \delta_0,\delta_1$ being now fixed, such a disk ${\rm exp}_x(\Gamma_g)$ is called an \emph{$F$-disk} at $x$.  In    Lemma \ref{Lem:Pliss-iterate}     the disks $\Delta$  and $\Delta_{f^nx}$   may be chosen to be $F$-disks and we will assume it is the case when applying this lemma in the following.    

\begin{Lemma}\label{disjoint}
Let $\Delta$ be an embedded disk.  There is $0<\delta_2<\delta_1$ such that if $\Delta'\subset \Delta$ and $\Delta''\subset \Delta$ are $F$-disks centered at $y'$ and $y''$,  and if $y'$ and $y''$  lie in a ball $B(x,\delta)$ with $\delta<\delta_2$,  then :
\begin{itemize}
\item either $\Delta'\cap B(x,\delta)=\Delta''\cap B(x,\delta)$,
\item or $\Delta'\cap \Delta''=\emptyset$ and there is $z'\in \Delta'\cap B(x,2\delta) $ and $z''\in \Delta''\cap B(x,2\delta)$ with $z\nsim^F z'$. 
\end{itemize}

\end{Lemma}
\begin{proof}
For $\delta_2\ll \delta_1$ small enough, ${\rm exp_x}^{-1}(\Delta')$ (resp.  ${\rm exp_x}^{-1}(\Delta')$ ) is a graph of a $2\theta_1$-Lipschitz maps $g'$  (resp. $g''$) from a  neighborhood of $F(x)^{\delta_1/2}$ to $E(x)$ and  $B(x,\delta_2)\cap \Gamma_{g'}\subset \Gamma_{g'|_{F(x)^{\delta_1/2}}}$ (resp.  idem for $g''$).  If these two graphs are not disjoint, they  coincide on their intersection because they are subdisks of $\Delta$.  Therefore $\Delta'\cap B(x,\delta_2)=\Delta''\cap B(x,\delta_2)$.  If they are disjoint a geodesic path $\gamma$ with $\gamma(0)=x$ and $\gamma'(0)\in  E(x)$ intersect transversally $\Delta'$ and $\Delta''$ in $B(x,2\delta)$.  If we let $z'$ and $z''$ these intersections then we have $z\nsim^F z'$.
\end{proof}

\subsection{F\o lner empirical measures}
By Fubini's theorem, we may find an  $F$-disk $D$ tangent to $\cC^{F}_{\theta}$ with $\theta$ as in Lemma \ref{Lem:Pliss-iterate} and 
with ${\rm Leb}_D(\{x\in U:  \ \overline{m}_F(x)>a''\})>0$. 

For a subset $E\subset\NN^*$,  we let  $\partial E$ be the symmetric difference of $E$ and $E+1$.  Given $m\in\NN$, for any subset $T\subset \NN$, denote by $T^m=\{n\in\NN:~\exists t_1\in T,~t_2\in T,~{\textrm s.t.}~t_1\le n\le t_2 \text{ and }~t_2-t_1\le m\}$.  
By applying \cite[Lemma  2]{Bur21} to the measurable set of Pliss times,  we get:

\begin{Theorem}\label{Thm:construction}
There are a sequence $\{n_\ell\}$ of positive integers, a sequence of subsets $\{\Lambda_{\ell}\}$ in $D\cap\{\overline{m}_F>a''\}$ indexed by $\ell\in  \NN$ such that
\begin{enumerate}
\item $\lim_{\ell\to\infty}n_\ell=\infty$;
\item\label{i.sub-exponential} The measure of $\Lambda_{\ell}$ satisfies:
$$\lim_{\ell\to\infty}\frac{1}{n_\ell}\log{\rm Leb}_D(\Lambda_{\ell})=0,$$
\item\label{i.density} For each $\ell$, there is a subset $Q_{\ell}\subset[0,n_\ell)$ s.t. 
$$\liminf_{\ell\to\infty}\frac{\#Q_{\ell}}{n_\ell}\ge \alpha,$$
where $\alpha$ is the Pliss density as in Lemma~\ref{Lem:Pliss}.
\item\label{i.boundary-condition} $\{Q_{\ell}\}$ has the F{\o}lner property:
$$\limsup_{\ell\to\infty}\frac{\#\partial Q_{\ell}}{n_\ell}=0$$
\item For any $x\in\Lambda_{\ell}$, one has that $\partial Q_{\ell}\subset P(x)$ for any $x\in\Lambda_{\ell}$.
\item $P(x)$ is uniformly dense in $Q_{\ell}$ on $\Lambda_{\ell}$, i.e.,
$$\limsup_{m\to\infty}\limsup_{\ell\to\infty}\sup_{x\in\Lambda_{l}}\frac{\#(Q_{\ell}\setminus P^m(x))}{n_\ell}=0$$
\item\label{i.uniform-lower} $\{Q_{\ell}\}$ has uniform lower density in the following sense:
$$\liminf_{\ell\to\infty}\inf_{x\in \Lambda_{\ell}}\frac{\#(P(x)\cap Q_{\ell})}{n_\ell}\ge\alpha.$$
\end{enumerate}
\end{Theorem}

Observe that  we get from Item~\ref{i.sub-exponential} and Item~\ref{i.density} that 
$\lim_{\ell\to\infty}\frac{1}{\# Q_\ell}\log {\rm Leb}_{D}(\Lambda_{\ell})=0.$

\subsection{F\o lner Gibbs Property}

Assume that $\cA$ is a finite partition of $M$. For each $x\in M$, denote by $\cA(x)$ the atom of  $\cA$ containing $x$.  Given two finite partitions $\cA$ and $\cB$, one defines the refinement of $\cA$ and $\cB$:
$$\cA\bigvee \cB=\{A\cap B:~A\in\cA,~B\in\cB\}.$$
One extends this definition to a finite sequence of partitions $\cA_1\bigvee\cA_2\bigvee\cdots\bigvee\cA_n$ inductively. For a finite subset $\MM\subset\NN\cup\{0\}$, we let
$$\cA^\MM=\bigvee_{n\in\MM}f^{-n}(\cA).$$

\smallskip

Let  $D$, $\{n_\ell\}$, $\{\Lambda_\ell\}$ and $\{Q_\ell\}$ as defined in the previous subsections.   The aim of this subsection is to prove the following Gibbs property: 
\begin{Theorem}\label{Thm:volume-estimate}
For any $\varepsilon>0$, there is $\delta>0$   such that for any partition $\cA$ of $M$ with ${\rm Diam}(\cA)<\delta$, there is $L\in\NN$ such that for any $\ell>L$, for any $x\in\Lambda_{\ell}$, one has that
$${\rm Leb}_{D}(\cA^{Q_{\ell}}(x)\cap\Lambda_{\ell})\le {\rm e}^{\varepsilon\# Q_{\ell}}{\rm e}^{-\sum_{i\in Q_\ell}\log{\rm Jac}\left(Df|_{F(f^ix)}\right)}.$$
\end{Theorem}

Recall the numbers $\theta_\varepsilon$ and $\delta_\varepsilon$ depending on $\varepsilon$ as in Lemma~\ref{Lem:estimate-at-bowen-ball}. We will prove the following intermediate proposition:
\begin{Proposition}\label{Pro:volume-estimate-any-manifold}
For any $\varepsilon$,   there are $\delta>0$  and $C>1$ satisfying the following property.

Let $n\in\NN$. Assume that $Q=\cup_{j=0}^{n-1}[a_j,b_j)$ satisfies $a_0=0$ and $a_{i+1}>b_i>a_i$. For any partition $\cA$ of $M$ with ${\rm Diam}(\cA)<\delta$, for any $x\in U$ and  for any $F$-disk $\Delta$ centered in $\cA^{Q}(x)$  and tangent to $\cC^{F}_{\theta_\varepsilon}$ we have with $\Delta^{Q}:=\{z\in\Delta:~a_j,b_j\in P(z),~\forall 0\le j\le n-1\}$ :
$${\rm Leb}_\Delta\left(\cA^{Q}(x)\cap\Delta^Q\right)\le C^{n}\cdot {\rm e}^{2\varepsilon\cdot b_{n-1}}{\rm e}^{-\sum_{j=0}^{n-1}\sum_{i=a_j}^{b_j-1}\log{\rm Jac}\left(Df|_{F(f^ix)}\right)}.$$

\end{Proposition}
\begin{proof}[Proof of Proposition~\ref{Pro:volume-estimate-any-manifold}]
We argue by induction on $n$. For $n=1$, the set of times $Q$ is reduced to an interval  $Q=[a_0,b_0)=[0,b_0)$.  By taking $C$ large enough we may assume $b_0>N$ with $N$ as in  Lemma~\ref{Lem:Pliss-iterate}.  Take $\Gamma=\cA^{Q}(x)\cap\Delta^Q=\{z\in\Delta:b_0\in P(z)\}\cap \cA^{[0,b_0)}(x)$.

\begin{Claim}
 $$\forall z\in \Gamma, \ \Gamma\subset f^{-b_0}(\Delta_{f^{b_0}z}).$$
\end{Claim}

\begin{proof}[Proof of the claim]
Let $z,z'\in \Gamma$.  For  $\delta$ small enough,  we have $d_{\Delta}(z,\partial \Delta), d_\Delta(z',\partial \Delta)>\delta_0/2$ and $d(f^kz,f^kz')<\delta_2/2$ for $k=0,\cdots, b_1$.  Recall that $\Delta_{f^{b_0}z}\subset f^{b_0}\Delta$ denotes  the $F$-disk  at $f^{b_0}(z)$   given by Lemma~\ref{Lem:Pliss-iterate}.  
Then by Lemma \ref{disjoint}
\begin{itemize}
\item  either 
 $\Delta_{f^{b_0}z}\cap  \cA(f^{b_0}x)=\Delta_{f^{b_0}z'}\cap  \cA(f^{b_0}x)$, then $f^{b_0}z'$  lies in $\Delta_{f^{b_0}z} $
\item or $\Delta_{f^{b_0}z}\cap \Delta_{f^{b_0}z'}=\emptyset$.
\end{itemize}
 In this last case there is a geodesic path transverse to $\cC_{\theta'}^F$ joining $f^{b_0}y\in \Delta_{f^{b_0}z}\cap B(f^{b_0}x,\delta_\varepsilon)$ and $f^{b_0}y'\in \Delta_{f^{b_0}z'}\cap B(f^{b_0}x,\delta_\varepsilon)$. As the  disks $\Delta_{f^{b_0}z'}$ and $\Delta_{f^{b_0}z}$ are backward contracting, we have  $d(f^ky,f^ky')<\delta_1$ for $k=0,\cdots, b_0$, whenever $\delta$ is small enough.  Because they lie in the same $F$-disc $\Delta$ we have $y\sim^Fy'$. It follows from the choice of $\delta_1$ that we have also $f^{b_0}y\sim^Ff^{b_0}y'$. This contradicts the fact that the geodesics joining these two points is transverse to $\cC_{\theta_0}^F$.

\end{proof}

As $\Delta$ is tangent to  $\cC_{\theta_{\varepsilon}}^F$,  we get for $\delta<\delta_\varepsilon$ by    Lemma~\ref{Lem:estimate-at-bowen-ball}:
\begin{align*}
{{\rm Leb}_\Delta(\Gamma)}&\le {{\rm Leb}_{f^{b_0}(\Delta)}\left(f^{b_0}(\Gamma)\right)}\cdot{\rm e}^{b_0\varepsilon}\cdot\left(\prod_{i=0}^{b_0-1}{\rm Jac}\left(Df|_{F(f^ix)}\right)\right)^{-1}.
\end{align*}
Then it follows from the above claim that for any $z\in \Gamma$:
\begin{align}\label{rape}
{{\rm Leb}_\Delta(\Gamma)}&\le {{\rm Leb}_{f^{b_0}(\Delta)}\left( \Delta_{f^{b_0}z}\right)}\cdot{\rm e}^{b_0\varepsilon}\cdot\left(\prod_{i=0}^{b_0-1}{\rm Jac}\left(Df|_{F(f^ix)}\right)\right)^{-1} \\
&\le C\cdot{\rm e}^{b_0\varepsilon}{\rm e}^{-\sum_{i=a_0}^{b_0-1}\log{\rm Jac}\left(Df|_{F(f^ix)}\right)},\nonumber
\end{align}
where the constant $C$ is an upper bound for the Lebesgue measure of $F$-disks. This concludes the case  $n=1$.

\smallskip

Now we assume the statement is true for $n$ and for any $F$-disk $\Delta$ and we check the statement for $n+1$. Let $Q=\cup_{j=0}^{n-1}[a_j,b_j)\cup [a_n,b_n)$,  $Q_1=\cup_{j=1}^n[a_j,b_j)$ and $Q_2=\cup_{j=1}^n[a_j-a_1,b_j-a_1)$ and let $\Delta$ be an $F$-disk.  Again, we may assume $b_j-a_j>N$ for any $j$.  Let $z \in\Delta\cap \cA^{[0,b_0)}(x)$  with  $b_0\in P(z)$.  By Lemma \ref{disjoint} there are finitely many points $y_1,y_2,\cdots,y_m\in\cA^{Q}(x)\cap\Delta^Q$ such that
\begin{itemize}
\item $\Delta_{f^{a_1}y_1},\Delta_{f^{a_1}y_2},\cdots,\Delta_{f^{a_1}y_m}\subset f^{a_1-b_0}(\Delta_{f^{b_0}z})$ are mutually disjoint;
\item 
$f^{a_1}\left(\cA^{Q}(x)\cap\Delta^Q\right)\subset  \bigcup_{i=1}^m \Delta_{f^{a_1}y_i}^{Q_2}$.
\end{itemize}


\begin{Claim}For some constant $C'$ depending on $\varepsilon$, we have 
\begin{equation}\label{e.sum-jacobiann}
\sum_{i=1}^m\left(\prod_{j=b_0}^{a_1-1}{\rm Jac}\left(Df|_{F(f^{j}y_i)}\right)\right)^{-1}\le C'{\rm e}^{\varepsilon(a_1-b_0)}.
\end{equation}
\end{Claim}
\begin{proof}[Proof of the Claim]
%
Since $a_1\in P(y_i)$,  by  (\ref{contractPl}),  there is $K\in \mathbb N$ depending  on $\varepsilon$,  such that  
$$f^{-a_1+b_0}(\Delta_{f^{a_1}y_i}) \subset \left\{y\in\Delta_{f^{b_0}z}:~d(f^iy,f^iy_i)<\delta_\epsilon,~b_0\le i\le \max(0, a_1-K)\right\}.$$
By changing the constant $C'$  we can assume $a_1>K$. 
Thus, by applying  Lemma~\ref{Lem:estimate-at-bowen-ball},  we get for some constants $B,B'$  depending on $\varepsilon$:
\begin{align}\label{rappe}
&~~~~~~{\rm Leb}_{\Delta_{f^{b_0}z} }\left(f^{b_0-a_1}\left(\Delta_{f^{a_1}y_i}\right)\right)\\
&\ge B {\rm e}^{-\varepsilon(a_1-b_0-K)}\left(\prod_{j=b_0}^{a_1-K}{\rm Jac}\left(Df|_{F(f^{j}y_i)}\right)\right)^{-1} \nonumber \\
&\ge B' {\rm e}^{-\varepsilon(a_1-b_0)}\left(\prod_{j=b_0}^{a_1-1}{\rm Jac}\left(Df|_{F(f^{j}y_i)}\right)\right)^{-1}.\nonumber
\end{align}
 By the mutual disjointness of $\{\Delta_{f^{a_1}y_i}\}_{i=1}^m$,  we conclude \begin{align*}
 D\geq {\rm Leb}\left(\Delta_{f^{b_0}z}\right) &\ge\sum_{i=1}^m {\rm Leb}_{\Delta_{f^{b_0}z}}\left(f^{b_0-a_1}(\Delta_{f^{a_1}y_i})\right)\\
&\ge  B' {\rm e}^{-\varepsilon(a_1-b_0)}\sum_{i=1}^m\left(\prod_{j=b_0}^{a_1-1}{\rm Jac}\left(Df|_{F(f^{j}y_i)}\right)\right)^{-1}.
\end{align*}
This concludes the proof of the claim with $C'=D/B'$, where $D$ is an upper bound of the volume of $F$-disks. 
\end{proof}
We are now in a position to conclude the proof by induction of Proposition \ref{Pro:volume-estimate-any-manifold}.  By the induction hypothesis,  we have for each $1\le i\le m$: 
\begin{align}\label{rapppe}
{\rm Leb}_{\Delta_{f^{a_1}y_i}}\left(\cA^{Q_2}(f^{a_1}x)\cap\Delta_{f^{a_1}y_i}^{Q_2}\right)&\le C^{n-1}{\rm e}^{2\varepsilon(b_n-a_1)}{\rm e}^{-\sum_{j=1}^{n}\sum_{i=a_j-a_1}^{b_j-a_1-1}\log{\rm Jac}\left(Df|_{F(f^{i+a_1}x)}\right)}\\
&=C^{n-1}{\rm e}^{2\varepsilon(b_n-a_1)}{\rm e}^{-\sum_{j=1}^{n}\sum_{i=a_j}^{b_j-1}\log{\rm Jac}\left(Df|_{F(f^{i}x)}\right)}.\nonumber
\end{align}

We conclude  with   some  constant $C''$ depending on $\varepsilon$ that
\begin{align*}
&{\rm Leb}_\Delta\left(\cA^{Q}(x)\cap\Delta^Q\right)\\
&\le {\rm Leb}_{\Delta_{f^{b_0}z}} \left( f^{b_0}(\cA^{Q}(x)\cap\Delta^Q) \right){\rm e}^{\varepsilon(b_0-a_0)}\left(\prod_{i=a_0}^{b_0-1}{\rm Jac}\left(Df|_{F(f^ix)}\right)\right)^{-1} \text{as in (\ref{rape})}\\
&\le\left(\sum_{i=1}^m{\rm Leb}_{\Delta_{f^{b_0}z}}\left(f^{b_0-a_1}\left(\cA^{Q_2}(f^{a_1}x)\cap\Delta_{f^{a_1}y_i}^{Q_2}\right)\right) \right)\\
&~~~~~~~~~~~~~~~~~~~~~~~~~~~~~~~~~~~~~~~~~~~~~~~~~~~~~~~~~\times{\rm e}^{\varepsilon(b_0-a_0)}\left(\prod_{i=a_0}^{b_0-1}{\rm Jac}\left(Df|_{F(f^ix)}\right)\right)^{-1} \text{by definition of $y_i$} \\
&\le C''\left( \sum_{i=1}^m {\rm Leb}_{\Delta_{f^{a_1}y_i}}\left(\cA^{Q_2}(f^{a_1}x)\cap\Delta_{f^{a_1}y_i}^{Q_2}\right)\left(
 \prod_{j=b_0}^{a_1-1} {\rm Jac}\left( Df|_{F(f^{j}y_i)} \right)\right)^{-1} \right)\\
&~~~~~~~~~~~~~~~~~~~~~~~~~~~~~~~~~~~~~~~~~~~~~~~~~~~~~~~~~\times {\rm e}^{\varepsilon(a_1-a_0)} \left(\prod_{i=a_0}^{b_0-1}{\rm Jac}\left( Df|_{F(f^ix)}\right)\right)^{-1} \text{ as in (\ref{rappe}})\\
&\le C''C^{n-1} {\rm e}^{\varepsilon\left(2(b_n-a_1)+(a_1-a_0)\right)} {\rm e}^{-\sum_{j=0}^{n}\sum_{i=a_j}^{b_j-1}\log{\rm Jac}\left(Df|_{F(f^{i}x)}\right)}  \sum_{i=1}^m\left( \prod_{j=b_0}^{a_1-1}{\rm Jac} \left(Df|_{F(f^{j}y_i)} \right)\right)^{-1} \text{ by (\ref{rapppe})}\\
&\le C^{n} {\rm e}^{2\varepsilon(b_n-a_0)} {\rm e}^{-\sum_{j=0}^{n}\sum_{i=a_j}^{b_j-1}\log{\rm Jac}\left(Df|_{F(f^{i}x)}\right)} \text{ by (\ref{e.sum-jacobiann}) and  by taking  $C\geq C'C''$.}
\end{align*}

\end{proof}

\begin{proof}[Proof of Theorem~\ref{Thm:volume-estimate}]
Without loss of generality we can assume that  $D$ is a $F$-disc tangent to $\cC_{\theta_\varepsilon}^F$ (indeed for $k$ large enough $f^k D$ is tangent to $\cC_{\theta_\varepsilon}^F$ and we may then cover $f^kD$ by finitely many $F$-disks).

Note that any integer in the boundary of $Q_\ell$ is a Pliss time for any point  $x\in \Lambda_\ell$. Thus, by Proposition~\ref{Pro:volume-estimate-any-manifold},  we get 
$${\rm Leb}_{D}(\cA^{Q_{\ell}}(x)\cap\Lambda_{\ell})\le C^{\partial Q_\ell}{\rm e}^{2\varepsilon n_\ell}{\rm e}^{-\sum_{i\in Q_\ell}\log{\rm Jac}\left(Df|_{F(f^ix)}\right)}.$$
By Item~\ref{i.boundary-condition} of Theorem~\ref{Thm:construction}, for $\ell$ large enough, one has that
$$C^{\partial Q_\ell}\le {\rm e}^{\varepsilon n_\ell}.$$
By combining with  Item~\ref{i.density} of Theorem~\ref{Thm:construction},  we get finally  for $\ell$ large enough
$${\rm Leb}_{D}\left(\cA^{Q_{\ell}}(x)\cap\Lambda_{\ell}\right)\le {\rm e}^{3\varepsilon n_\ell}{\rm e}^{-\sum_{i\in Q_\ell}\log{\rm Jac}\left(Df|_{F(f^ix)}\right)}.$$
\end{proof}

\subsection{Ergodic properties of the limit empirical measure  $\mu$}
To conclude the proof of Theorem \ref{Thm:dominated-entropy-formula} we just apply the abstract formalism established in Section 1 and 2 of \cite{Bur21}. 

We consider  a weak-$*$ limit $\mu$ of a sequence   $\left(\frac{1}{\sharp Q_\ell}\sum_{k\in Q_\ell}f_*^k\frac{\Leb_{D}(\cdot)}{\Leb_{D}(\Lambda_\ell)}\right)_\ell$ with $n_\ell\xrightarrow{\ell}+\infty$. It follows from the F\o lner Gibbs property proved in Theorem \ref{Thm:volume-estimate} and Proposition 3 in \cite{Bur21} that 
$$h(\mu)\geq \int \log {\rm Jac}\left(Df|_{F}\right)\, d\mu.$$

Observe also that for any  consecutive integers $k<l$ in $P(x)$ we have $$\sum_{k\leq i<l}\log m\left(Df|_{F(f^ix)}\right)\geq (l-k)a',$$  i.e.  following the terminology of \cite{Bur21} the set of Pliss time is $a'$-large with respect to the continuous observable $x\mapsto \log m\left(Df|_{F(x)}\right)$.  By Lemma 4 in \cite{Bur21},  the limit measure $\mu$ satisfies for $\mu$-a.e. $x$:
$$\overline{m}_F(x)=\underline{m}_F(x)\geq a'.$$

Equivalently,  for any ergodic component $\nu$  of $\mu$, we have 
$$\int \lambda_{{\rm dim}(F)}\, d\nu\geq  \int \log m(Df|_{F})\, d\nu\geq a'$$
By harmonicity of the Kolmogorov entropy,  there is at least one ergodic component $\nu$ of $\mu$,  such that  $h(\nu)\geq \int \log {\rm Jac}\left(Df|_{F}\right)\, d\nu.$
This ergodic measure $\nu$ satisfies the conclusion of Theorem \ref{Thm:dominated-entropy-formula}.

\section{The basin}
In this section we assume $f$ of class $\mathcal C^{1+\alpha}$ and we prove Corollary \ref{corConeplus}.  More precisely we will show that Lebesgue almost every point $x$ with $\chi_{min}^F(x)>\max(\overline{a},0)$ lies in the basin of an ergodic hyperbolic SRB measure.  To this end we adapt a standard argument involving the absolute continuity of Pesin stable manifolds  by introducing dynamical density points on Pliss times.

\subsection{Dynamical density points on Pliss times}\label{Sec:dynamical-density-pliss}

Pugh and Shub  introduced  in  \cite{PuS00}  dynamical density points for strong unstable manifold (one can also see \cite[Subsection 2.5]{CYZ20}). This notion generalizes the usual concept of Lebesgue density points in hyperbolic dynamics and is specially relevant   when the unstable manifolds have dimension larger than one because  the shape of an unstable  ball may be  not preserved by backward iteration.

In this subsection, we assume that $f$ is a $\mathcal C^{1+\alpha}$ diffeomorphism as above, i.e. $f$ admits an attractor   $\Lambda=\bigcap_{n\in \NN}f^n U$ with a dominated splitting $T_\Lambda M=E\oplus F$.  For  fixed $a''>a'>a>0$  and for $x \in U$ with $\overline{m}_F(x)>a''$  we consider the associated set of Pliss times $P(x)$ as in Subsection \ref{Subsection:Pliss}. Let $\Gamma$ be a subset of $\left\{\overline{m}_F>a''\right\}$ and let $D$ be a smooth disk tangent to $\cC_{\theta_0}^F$ with $\theta_0$ as in Lemma \ref{Lem:Pliss-iterate} satisfying ${\rm Leb}_D(\Gamma)>0$.  

For $n\in P(x)$  with $\overline{m}_F(x)>a''$  and for $\delta>0$ we define the following  dynamical balls: $B_{D,n}(x, \delta):=f^{-n}B_{f^n(D)}\left(f^nx, \delta\right)$.  To simplify the notations we write $B_{D,n}(x)$,  ${\widehat B}_{D,n}(x)$, ${\widehat{\widehat B}}_{D,n}(x)$ for   the dynamical balls $B_{D,n}\left(x, \frac{\delta_0}{3}\right)$, $B_{D,n}\left(x, \frac{2\delta_0}{3}\right)$ and $B_{D,n}\left(x, \delta_0\right)$ respectively.   
Then $x$ is said to be a \emph{dynamical density point of  $\  \Gamma$ with respect to $D$}
if
$$\lim_{n\to\infty,~n\in P(x)}\frac{{\rm Leb}_D\left(B_{D,n}(x)\cap \Gamma\right)}{{\rm Leb}_D\left(B_{D,n}(x)\right)}=1.$$

In this context,  Alves and Pinheiro \cite[Proposition 5.5]{AlP08} proved the existence of one dynamical density point $x$ with respect to $D$.   Here we improve this result as follows :

\begin{Theorem}\label{Thm:density-point}
${\rm Leb}_D$-a.e.   $x \in \Gamma$  is a dynamical density point of $\ \Gamma$ with respect to $D$.
\end{Theorem}

Before proving Theorem~\ref{Thm:density-point}, we state some properties of  the dynamical balls in our context. We omit the proofs,  which follow from   the uniform case \cite[Lemma 2.19]{CYZ20}.

\begin{Lemma}\label{Lem:double-size-disk}
There is $K>1$, such that for any  $z\in D$ with $\overline{m}_F(z)>a''$ and for any $n\in P(z)$,
$${\rm Leb}_D\left({\widehat{\widehat B}}_{D,n}(z)\right)\le K {\rm Leb}_D\left({B}_{D,n}(z)\right).$$

\end{Lemma}


\begin{Lemma}\label{Lem:intersection-contained}

For any $z_1,z_2\in \left\{\overline{m}_F>a''\right\}$, for any $n_1\in P(z_1)$ and $n_2\in P(z_2)$ satisfying $n_1\le n_2$, if $B_{D,n_1}(z_1)\cap {\widehat B}_{D,n_2}(z_2)$, then ${B}_{D,n_2}(z_2)\subset {\widehat{\widehat B}}_{D,n_1}(z_1)$.
\end{Lemma}

\begin{proof}[Proof of Theorem~\ref{Thm:density-point}]
It suffices to prove that for ${\rm Leb}_D$-a.e. $x\in \Gamma$, one has that
$$\liminf_{n\to\infty,~n\in P(x)}\frac{{\rm Leb}_D(B_{D,n}(x)\cap\Gamma)}{{\rm Leb}_D(B_{D,n}(x))}=1.$$
Given $\rho\in(0,1)$, we let 
$$\Gamma_\rho:=\left\{z\in\Gamma:~\liminf_{n\to\infty,~n\in P(x)}\frac{{\rm Leb}_D\left(B_{D,n}(x)\cap\Gamma\right)}{{\rm Leb}_D\left(B_{D,n}(x)\right)}<\rho\right\}.$$
It is enough to show   ${\rm Leb}_D(\Gamma_\rho)=0$ for any $\rho\in(0,1)$.  We argue by contradiction by considering some  $\rho\in(0,1)$ with ${\rm Leb}_D(\Gamma_\rho)>0$. Fix  $\varepsilon>0$ with $\rho(1+\varepsilon)<1$. We choose an open neighborhood $V$ of $\Gamma_\rho$ in $D$ such that
$${\rm Leb}_D(V)<(1+\varepsilon){\rm Leb}_D(\Gamma_\rho).$$
We consider the  covering $\mathcal V$ of $\Gamma_\rho$ defined as follows : 
$${\mathcal V}:=\left\{B_{D,n}(z)\subset V:~z\in\Gamma_\rho,~n\in P(z),~\frac{{\rm Leb}_D\left(B_{D,n}(x)\cap\Gamma\right)}{{\rm Leb}_D\left(B_{D,n}(x)\right)}<\rho\right\}.$$

We build by induction  an increasing sequence of integers $(n_\ell)_\ell$,   subfamilies  $(\mathcal V_\ell)_\ell$ of $\mathcal V$ and finite subsets $(Z_\ell)_\ell$ of $\Gamma_\rho$  such that 
\begin{enumerate}[(i)]
\item the elements of $\bigcup_\ell \mathcal V_l$ are pairwise disjoint, 
\item  for any $\ell$ the elements of $\mathcal V_\ell$ are  of the form  $B_{D,n_\ell}(z^\ell)$ with $z_\ell\in Z_\ell$ and $n_\ell\in P(z^\ell)$,
\item if $B_{D,n}(z)\in \mathcal V $ with $n_{\ell-1}\leq n< n_\ell$,    then $\widehat B_{D,n}(z)\cap B\neq \emptyset$ for some $B\in \bigcup_{\ell'< \ell} \mathcal V_{\ell'}$. 
\end{enumerate}

\paragraph{Initialization:} We find a minimal integer $n_1$ such that there is $z_1^1\in \Gamma_\rho$ with $n_1\in P(z_1^1)$ satisfying $B_{D,n_1}(z_1^1)\in {\mathcal V}$. After $z_1^1$, we choose $z_2^1$ such that $B_{D,n_1}(z_1^1)\cap {\widehat B}_{D,n_1}(z_2^1)=\emptyset$. If no such $z_2^1$, we stop this process; if we can find such $z_2^1$, then we choose $B_{D,n_1}(z_2^1)\in {\mathcal V}$.  Assume that $\{B_{D,n_1}(z_1^1),B_{D,n_1}(z_2^1),\cdots,B_{D,n_1}(z_j^1)\}$ has been found. We choose $z_{j+1}^1$ such that $B_{D,n_1}(z_i^1)\in \mathcal V$ and $B_{D,n_1}(z_i^1)\cap {\widehat B}_{D,n_1}(z_{j+1}^1)=\emptyset$ for any $1\le i\le j$. If no such $z_{j+1}^1$, we stop this process.  We let $k(1)$ be the stopping time 
and we put 
 $$Z_1:=\left\{z_1^1,z_2^1,\cdots,z_{k(1)}^1\right\}$$ and
$$\cV_1:=\left\{B_{D,n_1}\left(z_1^1\right),B_{D,n_1}\left(z_2^1\right),\cdots,B_{D,n_1}\left(z_{k(1)}^1\right)\right\}.$$

\paragraph{Induction step:}

Assume $n_1<n_2<\cdots<n_\ell$, $Z_1,Z_2,\cdots,Z_\ell$ and $\cV_1,\cV_2,\cdots,\cV_\ell$ have been found.  We choose a minimal $n_{\ell+1}>n_\ell$ such that there is $z_1^{\ell+1}$ satisfying $B_{D,n_{\ell+1}}(z_1^{\ell+1})\in {\mathcal V}$, but ${\widehat B}_{D,n_{\ell+1}}(z_1^{\ell+1})\cap B=\emptyset$ for any $B\in {\mathcal V}_1\cup {\mathcal V}_2\cup\cdots\cup{\mathcal V}_\ell$.  Assume that $z_1^{\ell+1},z_2^{\ell+1},\cdots,z_j^{\ell+1}$ have been fixed.  We choose $z_{j+1}^{\ell+1}$ such that $B_{D,n_{\ell+1}}(z_i^{\ell+1})\in \mathcal V$ and  $B_{D,n_{\ell+1}}(z_i^{\ell+1})\cap {\widehat B}_{D,n_{\ell+1}}(z_{j+1}^{\ell+1})=\emptyset$ for any $1\le i\le j$ and ${\widehat B}_{D,n_{\ell+1}}(z_{j+1}^{\ell+1})\cap B=\emptyset$ for any $B\in {\mathcal V}_1\cup {\mathcal V}_2\cup\cdots\cup{\mathcal V}_\ell$. If no such $z_{j+1}^{\ell+1}$, we stop the process.  We let $k(\ell+1)$ be the stopping time 
and we put $$Z_{\ell+1}:=\left\{z_1^{\ell+1},z_2^{\ell+1},\cdots,z_{k({\ell+1})}^{\ell+1}\right\}$$ and 
$$\cV_{\ell+1}:=\left\{B_{D,n_{\ell+1}}(z_1^{\ell+1}),B_{D,n_{\ell+1}}(z_2^{\ell+1}),\cdots,B_{D,n_{\ell+1}}(z_{k({\ell+1})}^{\ell+1})\right\}.$$

\begin{Claim}
${\rm Leb}_D(\widetilde\Gamma_\rho)=0$, where $\widetilde\Gamma_\rho:=\Gamma_\rho\setminus\left(\bigcup_{n=1}^\infty\bigcup_{B\in \cV_n}B\right)$.

\end{Claim}
\begin{proof}[Proof of the Claim]Fix some integer $\ell$. 
For $z\in\widetilde\Gamma_\rho$,  the diameter of  $B_{D,n}(z)\in \mathcal V$  goes to $0$ when  $n\in P(z)$  goes to infinity. In particular for $n$ large enough we have 
$$\forall B\in \bigcup_{\ell' \leq \ell}\mathcal V_{\ell'},  \ B\cap \widehat B_{D,n}(z)=\emptyset.$$
Fix such an integer $n=n(z,\ell)$. 
Together with (iii) there is $\ell''>\ell$ with $n_{\ell''}\leq n$ and $B\in \mathcal V_{\ell''}$ such that  $$B\cap 
\widehat B_{D,n}(z)\neq\emptyset.$$ By Lemma~\ref{Lem:intersection-contained},  we get  
$B_{D,n}(z)\subset {\widehat{\widehat B}}_{D,n}(z)$,  so that 

$${\widetilde\Gamma}_\rho\subset \bigcup_{k=\ell+1}^\infty\bigcup_{z\in Z_k}{\widehat{\widehat B}}_{D,n_k}(z).$$
By Lemma~\ref{Lem:double-size-disk}, there is $K>1$ such that 
$${\rm Leb}_D\left({\widehat{\widehat B}}_{D,n_k}(z)\right)\le K{\rm Leb}_D\left(B_{D,n_k}(z,\delta)\right).$$
Since the elements in $\bigcup_k\cV_k$ are mutually disjoint,  we have  $\sum_{k\in\NN}\sum_{z\in Z_k}{\rm Leb}_D(B_{D,n_k}(z))<\infty$. Thus 
$${\rm Leb}_D(\widetilde\Gamma_\rho)\le\sum_{k=\ell+1}^\infty\sum_{z\in Z_k}{\rm Leb}_D\left({\widehat{\widehat B}}_{D,n_k}(z)\right)\le K\sum_{k=\ell+1}^\infty\sum_{z\in Z_k}{\rm Leb}_D(B_{D,n_k}(z))\to 0,~\textrm{as}~\ell\to\infty.$$
This implies that ${\rm Leb}_D(\widetilde\Gamma_\rho)=0$ and completes the proof of the claim.
\end{proof}

It follows from the above Claim,   that
\begin{align*}
&~~~~~~~{\rm Leb}_D(\Gamma_\rho)=\sum_{n=1}^\infty\sum_{B\in V_n}{\rm Leb}_D(B\cap \Gamma_\rho)\le \sum_{n=1}^\infty\sum_{B\in V_n}{\rm Leb}_D(B\cap \Gamma)\\
&\le \rho\cdot\sum_{n=1}^\infty\sum_{B\in V_n}{\rm Leb}_D(B)\le \rho\cdot{\rm Leb}_D(V)\le \rho(1+\varepsilon){\rm Leb}_D(\Gamma_\rho).
\end{align*}
As we fixed  $\varepsilon$ with  $ \rho(1+\varepsilon)<1$ we conclude ${\rm Leb}_D(\Gamma_\rho)=0$. The proof is thus complete.
\end{proof}

We will use the following corollary of Theorem~\ref{Thm:density-point}, which is a direct  application of Egorov's theorem.

\begin{Corollary}\label{Cor:density-point-uniform-Egorov}
For any $\varepsilon>0$, there is a measurable set $\Gamma_0\subset \Gamma$ satisfying ${\rm Leb}_D(\Gamma\setminus\Gamma_0)<\varepsilon$ such that for any $\delta>0$, there is $N=N(\delta)$ such that for any $x\in\Gamma_0$ for any $n>N$, $n\in P(x)$,  it holds  that
$$\left| \frac{{\rm Leb}_D(B_{D,n}(x)\cap\Gamma)}{{\rm Leb}_D(B_{D,n}(x))}-1\right|<\delta.$$
\end{Corollary}

\subsection{Proof of Corollary \ref{corConeplus}}  It is enough to show that any subset $\Gamma$ of  $\left\{\overline{m}_F>\max(\overline{a},0)\right\}$ with positive Lebesgue measure  has a nonempty intersection with the basin of an ergodic SRB hyperbolic measure  supported on $\Lambda$. 

We may now  repeat the construction  in Theorem~\ref{Thm:dominated-entropy-formula}.  Let us recall the main steps. Without loss of generality we may assume $\Gamma\subset \left\{\overline{m}_F>a''\right\}$ for some  $a''>a'>\max(\overline{a},0)$. Take  a disk  $D$ tangent to $\cC^{F}_{\theta_0}$ with
${\rm Leb}_D(\Gamma)>0 $ 
Recall $P(x)$ denotes the set of Pliss times and  consider the sets $\Lambda_\ell\subset M$ and $Q_\ell\subset \NN$ as in  Theorem~\ref{Thm:construction}.  Without loss of generality we can assume $\Lambda_\ell\subset  \Gamma_0$ where $\Gamma_0$ is the set given by Corollary \ref{Cor:density-point-uniform-Egorov}. 
Set $$\mu_{\ell}=\frac{\Leb_{D}(\cdot)}{\Leb_{D}(\Lambda_\ell)},$$
$$\nu_\ell=\frac{1}{\# Q_{\ell}}\sum_{i\in Q_{\ell}}f^i_*\mu_{\ell},$$
and 
$$\eta_\ell=\frac{1}{\# Q_{\ell}} \int \sum_{i\in P(x)\cap Q_\ell}\delta_{f^ix} \ d\mu_{\ell}.$$

By taking a subsequence we may assume $\nu_\ell$ and $\eta_\ell$ are both converging to $\mu$ and $\eta$ respectively, when $\ell$ goes to infinity. Clearly $\eta$ is a component (in general not invariant) of  the invariant measure $\mu$ and by  Item~\ref{i.uniform-lower} of Theorem~\ref{Thm:construction} we have 
\begin{align*}
\eta(M)&\geq \lim_\ell \eta_\ell(M),\\
&\geq  \liminf_{\ell\to\infty}\frac{\inf_{x\in\Lambda_\ell}\# P(x)\cap \Lambda_\ell}{\# Q_\ell},\\
&\geq\alpha.
\end{align*}

From the proof of Theorem~\ref{Thm:dominated-entropy-formula},   any limit $\mu$ of $(\mu_\ell)_\ell$ satisfies

 \begin{equation*} \lambda_{\mathrm{dim}(F)} >\max(\overline{a},0) \ \mu\text{-a.e.}
\end{equation*}
and  \begin{equation*}h_\mu(f)\ge \int\log{\rm Jac}(Df|_{F}) d\mu.
 \end{equation*}
 By definition of $\overline{a}$,  an ergodic component  $\xi$ of $\mu$ either satisfies $h_\xi(f)<\int\log{\rm Jac}(Df|_{F}) d\xi$ or 
 $i_{cu}(\xi)={\rm dim}(F)$, i.e. $\lambda_{{\rm dim}(F)+1}(x)<0$ for $\xi$-a.e. $x$. In this last case we have also by Ruelle inequality $h_\xi(f)\leq \int\sum_{\lambda_i(x)>0}\lambda_i(x)\ d\xi(x)=\int\log{\rm Jac}(Df|_{F}) d\xi$. But $h_\mu(f)\ge \int\log{\rm Jac}(Df|_{F}) d\mu$, so that  the first case does not occur and thus any ergodic component $\xi$ of $\mu$ is an ergodic SRB hyperbolic component. As already mentioned these measures are physical, thus there are at most countably many of them. The measure $\mu$ is a convex combination of these  ergodic hyperbolic SRB measures $(\mu^i)_{i\in \mathbb N}$, i.e. there is $\lambda_i>0$, $i\in \NN$,  with $\mu=\sum_i\lambda_i\mu^i$. 
There is a Pesin block $R$ with $\eta(R)>0$ such that the  size of stable manifolds at $x \in R$ has uniform size
and these stable manifolds on $R$ form an absolutely continuous 
foliation. Let $i$ with $\mu^i(R)>0$. Let $W^u(x)$ denotes the Pesin local unstable manifold at a $\mu^i$-typical point $x$. We 
denote the basin of $\mu^i$ by $\mathcal B(\mu^i)$.  By the 
geometric property of the SRB measure $\mu^i$, the set $A_R^i:=\left\{x, \ {\rm Leb}_{W^u(x)}\left(\mathcal B(\mu^i)\cap R\right)>0\right\}$ has 
positive $\mu^i$-measure, therefore positive $\eta$-measure.  It 
follows from the definition of $\eta$ that there is $x\in A^i_R$,  $y_{\ell}\in \Lambda_\ell$ and $m_\ell \in P(y_\ell)$ for  infinitely many $\ell$ such that $m_\ell$ goes to infinity and 
 $f^{m_\ell}y_\ell$ goes to $x$ when $\ell$ goes to infinity.  Observe that $W^u(x)$ is tangent to $F$ and $f^{m_\ell}D$ is tangent to $
\cC_{\theta_0/2^{m_\ell}}^F$.  Let $E^i_R=W^u(x)\cap\mathcal B(\mu^i)\cap R$ and  $W^s\left(E_R^i\right):=\bigcup_{x\in E_R^i }W^s(x)$. By the absolutely continuity of the stable 
foliation on $R$ we have with the notations of the previous subsection  $${\rm Leb}_{f^{m_\ell}D}\left(W^s\left(E_R^i\right)\cap f^{m_\ell}B_{D,m_\ell}(y_\ell)  \right)>B {\rm Leb}
_{W^u(x)}\left(E_R^i\right)$$ for some constant $B>0$ depending on $R$. Then by the distorsion property  (\ref{equation:distorsion}) we have for some other constant $B'>0$ 
\begin{align*}\frac{{\rm Leb}_{D}\left(  f^{-m_\ell}W^s(E_R^i)\cap B_{D,m_\ell}(y_\ell)  \right)}{{\rm Leb}_{D}\left( B_{D,m_\ell}(y_\ell) \right)}&>B'\frac{{\rm Leb}_{f^{m_\ell}D}\left(W^s\left(E_R^i\right)\cap f^{m_\ell} B_{D,m_\ell}(y_\ell)  \right)}{{\rm Leb}_{f^{m_\ell}D}\left(f^{m_\ell} B_{D,m_\ell}(y_\ell) \right)}\\
&>BB' {\rm Leb}_{W^u(x)}(E_R^i).
\end{align*}
But as $y_l\in \Gamma_0$ we have for $\ell$ large enough:
$$ \frac{{\rm Leb}_D\left(B_{D,m_\ell}(y_\ell)\cap\Gamma\right)}{{\rm Leb}_D\left(B_{D,m_\ell}(y_\ell)\right)}>1-BB' {\rm Leb}_{W^u(x)}\left(E_R^i\right),$$
therefore 
$$ \Gamma \cap f^{-m_\ell}W^s(E_R^i)\neq \emptyset.$$
  This concludes the proof as $ f^{-m_\ell}W^s\left(E_R^i\right)$ is contained in  the basin of $\mu^i$.

\section{Finiteness}
In this last section we prove Theorem \ref{Thm:ergodic-finiteness-away-from-zero} about the finiteness of SRB measures. 
\subsection{Plaque family theorem and stable manifolds}
We first recall some standard facts. The plaque family theorem from \cite[Theorem 5.5]{HPS77} states as follows:

\begin{Theorem}\label{Thm:plaque}

Assume that $\Lambda$ is a compact invariant set with a dominated splitting $T_\Lambda M=E\oplus F$,   then  for every $x \in \Lambda$,  there exists a $\mathcal C^1$ embedding $\psi^F_x : B(0,1) \subset F(x)\rightarrow M$ with the following properties:
\begin{itemize}
\item $ \psi^F_x(0)  = x$ and the image of $\psi^F_x$ is tangent to F(x) at x.
\item the embeddings  $\psi^F_x$ depend continuously on the $\mathcal C^1$ topology on $x\in \Lambda$.
\item there is  $\delta<1$ such that $f^{-1}\psi^F_x\left(B(0,\delta)\right)\subset \psi^F_x\left(B(0,1)\right)$ for any $x\in \Lambda$ .
\end{itemize}
A similar conclusion holds for the bundle $E$.
\end{Theorem}
The following lemma can be found in many papers, for instance \cite[Section 8.2]{ABC08} or \cite{CMY22}.
\begin{Lemma}\label{Lem:stable-manifold-plaque}
Assume that $\Lambda$ is a compact invariant set with a dominated splitting $T_\Lambda M=E\oplus F$.
For any $\gamma\in(0,1)$ and any $K>0$, there is $\delta=\delta(\gamma,K)$ such that for any $x\in\Lambda$, if 
$$\prod_{i=0}^{n-1}\|Df^{-1}|_{F(f^{-i}x)}\|\le K\lambda_2^n,~~~\forall n\in\NN$$
then $\psi^F_x\left(B(0,\delta)\right)$  is contained in the unstable manifold of $x$.
\end{Lemma}

\subsection{Bi-Pliss times}
We assume that $f$ is a $\mathcal C^{1+\alpha}$ diffeomorphism.
Given a compact invariant set $\Lambda$, denote by $\cM_{{\rm SRB}}(\Lambda)$ the set of ergodic SRB measures supported on $\Lambda$. When $\Lambda$ admits a dominated splitting $T_\Lambda M=E\oplus F$ and $\mu$ is an ergodic measure $\mu$ supported on $\Lambda$, we let 
$$m_F(\mu)=\int \log m(Df|_{F}){\rm d}\mu,~~~M_E(\mu)=\int \log\|Df|_{E}\|{\rm d}\mu.$$
By Birkhoff Ergodic theorem, we have  for $\mu$-almost every point $x$:
$$m_F(\mu)=\lim_{n\to+\infty}\frac{1}{n}\sum_{i=1}^{n}\log m\left(Df|_{F(f^ix)}\right)=\lim_{n\to+\infty}-\frac{1}{n}\sum_{i=0}^{n-1}\log \left\|(Df^{-1}|_{F(f^{-i}x)}\right\|,$$
$$M_E(\mu)=\lim_{n\to\infty}\frac{1}{n}\sum_{i=0}^{n-1}\log \left\|Df|_{F(f^ix)}\right\|.$$
 Given $K>0$ and $\gamma\in(0,1)$, we define the following Pesin blocks 
$$P^{F}(\gamma,K):=\left\{x\in\Lambda:~\prod_{i=0}^{n-1}\left\|Df^{-1}|_{F(f^{-i}x)}\right\|\le K\gamma^n,~\forall n\in\NN\right\},$$
and 
$$P^{E}(\gamma,K):=\left\{x\in\Lambda:~\prod_{i=0}^{n-1}\left\|Df|_{E(f^ix)}\right\|\le K\gamma^n,~\forall n\in\NN\right\}.$$

To simplify the notations we just write $P^{F}(\gamma)$ and  $P^{E}(\gamma)$ for $P^{F}(\gamma,1) $ and  $P^{E}(\gamma,1)$   respectively.

The mean ergodic inequality states that if $\phi$ is an integrable observable of an ergodic measure preserving system $(X,T, \mathcal B, \nu)$ then $$\int_A \phi\, d\nu\geq 0$$ with $$A:=\left\{x\in X, \ \exists n>0 \text{ with }\sum_{k=0}^{n-1}\phi\circ T^kx>0\right\}.$$
In particular \begin{equation}\label{equ:mean}
\int_{X\setminus A} \phi\, d\nu\leq \int \phi\,d\nu. \end{equation}
 If $\mu$ is an ergodic measure of $f$ with $ m_F(\mu)>-\log \gamma$, we get $M\setminus A=P^F(\gamma)$ with $\phi=\log \|Df^{-1}|_{F}\|-\log \gamma$ and Equation (\ref{equ:mean})  for the standard Borel $\sigma$-algebra $\mathrm{Bor}$ on $M$ and the measure preserving system $(M,f^{-1}, \mathrm{Bor}, \mu)$ reads as follows:

$$\int_{P^F(\gamma)} \phi \, d\mu\leq \int \phi\, d\mu= m_F(\mu)-\log \gamma <0,$$
 in particular $$\mu\left(P^F(\gamma)\right)>0.$$

 We consider an adapted norm for the dominated splitting (see \cite{Gou}), in particular 
 \begin{equation}\label{equ: adapted}\forall x\in \Lambda, \ \|Df|_{E(x)}\|\|Df^{-1}|_{F(fx)}\|\leq \lambda<1.  
 \end{equation}

\begin{Lemma}\label{Lem:bi-Pliss}
Let $\gamma \in (\lambda^{1/2},1)$ and  let  $\mu$ be an   ergodic measure satisfying $m^{F}(\mu)>-\log\gamma$ and $M^E(\mu)<\log\gamma$. Then

$$\mu\left(P^F(\gamma)\cap P^E(\gamma)\right)=\mu\left(P^F(\gamma)\right)=\mu\left(P^E(\gamma)\right)>0.$$

\end{Lemma}

\begin{proof}

Take $x$ typical for $\mu$ so that we have   with $A=P^F(\gamma)$, $ P^E(\gamma)$ or   $P^E(\gamma)\cap P^F(\gamma)$:
$$\frac{1}{|n|}\sharp \left\{k\in [0,n), \ f^kx \in A\right\}\xrightarrow{n\rightarrow \infty}\mu(A).$$
To prove the lemma, it is enough to show $\left[f^kx\in P^F(\gamma)\right]\Leftrightarrow \left[f^kx \in P^E(\gamma)\right]$  for any $k\in \mathbb Z$. In fact, by symmetry under taking the inverse of $f$, we only need  to prove  the implication $\Rightarrow$. As $\mu\left(P^F(\gamma)\right)>0$ (resp. $\mu\left(P^E(\gamma)\right)>0)$,  for $\mu$-almost every point, there are infinitely many $n>0$ and infinitely many $n<0$ such that $f^nx\in P^{F}(\gamma)$ (resp.  $\in P^E(\gamma)$). 
We argue by contradiction by assuming $\Rightarrow$ does not hold true. Let $k$ be an integer  with $f^kx\in P^F(\gamma)\setminus P^E(\gamma)$.  We consider the smallest integer $m$ larger than $k$ with $f^mx\in P^{E}(\gamma)$. Then there are 
there are integers $n_1$ and $n_2$ with  $k\leq n_1<m\leq n_2$ such that 
\begin{enumerate}[(i)]
\item  $f^{n_i}x\in P^{F}(\gamma)$ for $i=1,2$;
\item for any $n\in (n_1,n_2)$, one has that $f^nx\notin P^{F}(\gamma)$.
\end{enumerate}
By definition of $m$ we have $f^{n_1}x\in P^F(\gamma)\setminus P^E(\gamma)$.

\begin{Claim}

$$\forall n_1<j\le n_2-1, \ \prod_{i=n_1+1}^j {\left\|Df^{-1}|_{F(f^ix)}\right\|}>\gamma^{j-n_1}.$$
\end{Claim}
\begin{proof}[Proof of the Claim]
The Claim holds for $j=n_1+1$. Otherwise,  $\|Df^{-1}|_{F(f^{n_1+1}x)}\|\le\gamma$. Since $f^{n_1}(x)\in P^{F}(\gamma)$, by concatenation, we get $f^{n_1+1}x\in P^{F}(\gamma)$. This contradicts  the second item (ii).

Now we assume that the formula holds for $n_1+1,n_1+2,\cdots,j<n_2-1$, i.e., for any $n_1+1\le \ell\le j$, we have 
\begin{equation}\label{e.long-bad}
\prod_{i=n_1+1}^\ell {\left\|Df^{-1}|_{F(f^ix)}\right\|}>\gamma^{\ell-n_1}.
\end{equation}
We will check the claim for $j+1$. If it does not hold, then 
\begin{equation}\label{e.one-good}
\prod_{i=n_1+1}^{j+1} {\left\|Df^{-1}|_{F(f^ix)}\right\|}\le\gamma^{j+1-n_1}.
\end{equation}
Thus, by \eqref{e.long-bad} and \eqref{e.one-good} for any $n_1+1\le \ell\le j$, one has 
$$\prod_{i=\ell+1}^{j+1} {\left\|Df^{-1}|_{F(f^ix)}\right\|}\le\gamma^{j+1-\ell}.$$
Since $f^{n_1}x\in P^{F}(\gamma)$, by concatenation, we get $f^{j+1}x\in P^{F}(\gamma)$. This contradicts  the second item (ii).
\end{proof}
It follows from the above claim, the domination property (\ref{equ: adapted}) and $\gamma^2 >\lambda$ that for any $n_1<j\le n_2-1$:
\begin{align*}
\prod_{i=n_1}^{j-1} \left\|Df|_{E(f^ix)}\right\|\le \lambda^{j-n_1}\left(\prod_{i=n_1+1}^j {\left\|Df^{-1}|_{F(f^ix)}\right\|}\right)^{-1}\le\lambda^{j-n_1}/\gamma^{j-n_1}\le \gamma^{j-n_1}.
\end{align*}
Then by concatenating up to $m$, we obtain the contradiction  $f^{n_1}x\in P^{E}(\gamma)$.
\end{proof}

\subsection{Proof of Theorem \ref{Thm:ergodic-finiteness-away-from-zero}}

Recall that we have by the geometric property of SRB measures

\begin{Fact}
For a hyperbolic ergodic SRB measure $\mu$, for $\mu$-almost every point $x$, ${\rm Leb}_{W^u(x)}$-almost every point $y$ is contained in the basin of $\mu$.
\end{Fact}

We are now in a position to prove Theorem \ref{Thm:ergodic-finiteness-away-from-zero}. 
Fix $a>0$ and write   $a=-\log \gamma$ with $\gamma\in (0,1)$.

Let $\mu$ be an ergodic (hyperbolic) SRB measure with $m^{F}(\mu)>-\log\gamma$ and $M^E(\mu)<\log\gamma$.  For $\mu$-typical points $x$,  the unstable bundle at $x$ coincides with $F(x)$.  If $x$ belongs moreover to  $ P^F$  the size of the unstable manifold at $x$ is bounded from below by some $\delta>0$ by  Lemma \ref{Lem:stable-manifold-plaque} (we let $W_\delta^u(x)=\psi^F_x\left(B(0,\delta)\right)$ in the following). 

From Lemma \ref{Lem:bi-Pliss}, by taking $\mathcal B(\mu)$ to be the basin of $\mu$, we get 
$$\mu\left(P^F(\gamma)\cap P^E(\gamma)\right)=\mu\left(P^F(\gamma)\cap P^E(\gamma)\cap \mathcal B(\mu) \right)>0.$$
By the geometric property of SRB measures there is $x=x_\mu\in P^F(\gamma)$ such that 
$${\rm Leb}_{W^u_\delta(x)}\left(P^E(\gamma)\cap \mathcal B(\mu)\right)>0$$
Without loss of generality we may assume $x_\mu$ is a Lebesgue density point of $ P^E(\gamma)\cap \mathcal B(\mu)$ for ${\rm Leb}_{W^u_\delta(x)}$. 

 Assume there are infinitely many distinct such  SRB measures $\mu_1,\cdots ,\mu_n, 
 \cdots$ and let $x_n=x_{\mu_n}$ for any $n$.  As the stable (resp.  unstable) manifolds at 
 points in $P^F(\gamma)$ (resp.  in $P^E(\gamma)$) have lower bounded size,     $W^s(y)$ is transverse to 
 $W^u_\delta(x_n)$ for $y\in W^u_\delta(x_m) \cap P^E(\gamma)$ whenever $x_n$ and $x_m$, $ n\neq m$,    are close enough.  By absolutely continuity of the stable foliation, 
 we get 
 
$${\rm Leb}_{W^u_\delta(x_n)}\left( W^s\left( W^u_\delta(x_m)\cap P^F(\gamma)\cap P^E(\gamma)\cap \mathcal B(\mu_m) \right)\right)>0.$$

From the above fact the basins  $\mathcal B(\mu_n ) $ and  $\mathcal B(\mu_m) $ have non  empty intersection, therefore $\mu_n=\mu_m$.  This contradicts our assumption and concludes the proof of 
Theorem \ref{Thm:ergodic-finiteness-away-from-zero}.

\section*{Appendix}Following Lemma 1 in \cite{Bur19}, we prove the equality (\ref{eq:empexp}) stated in the introduction. 

\begin{Lemma}Assume that $\Lambda=\bigcap_{n\in \NN}f^nU$ is an attractor  of  a $\mathcal C^{1}$ diffeomorphism $f$ with a dominated splitting $T_\Lambda M=E\oplus  F$.  Then we have for all $x\in U,$  $$\chi^F_{\mathrm{min}}(x,f) =\sup_{\mu\in pw(x)}\int \lambda_{\mathrm{dim}(F)}\, d\mu. $$
\end{Lemma}

\begin{proof}
For $x\in U$ and $p\in \mathbb N^*$ we let $\beta_p(x):=\limsup_{n\to+\infty}\frac{1}{n}\sum_{i=0}^{n-1}\log m\left(Df^p|_{F(f^ix)}\right)$.  The sequence 
$\left(\beta_p(x)\right)_p$ is superadditive,  in particular  the limit $\chi_{\mathrm{min}}^F(x,f)=\lim_{p\to+\infty} \frac{\beta_p(x)}{p}$ is well defined.  
Let $\mu=\lim_k\mu^{n_k}_x\in p\omega(x)$ for an increasing sequence of integers $(n_k)_k$. 
For all positive integers $n$ and $p$ we have 
\[\int \log m\left(Df^p|_{F(y)}\right)\, d\mu_n^x(y)=\frac{1}{n}\sum_{i=0}^{n-1}\log m\left(Df^p|_{F(f^ix)}\right).\]
Taking the limit over $n=n_k$ when $k$ goes to infinity  we get 
\begin{equation}\label{eq::} \int\frac{ \log m\left(Df^p|_{F(y)}\right)}{p}\, d\mu(y)\leq \frac{\beta_p(x)}{p}.
\end{equation}

It is well known that for $\mu$ a.e.  $y$,  
$$\frac{ \log m\left(Df^p|_{F(y)}\right)}{p}\xrightarrow{p\to +\infty}  \lambda_{\mathrm{dim}(F)}(y).$$ 
Therefore by taking the limit when $p$ goes to infinity in  (\ref{eq::}) we have finally
\[\sup_{\mu\in pw(x)}\int \lambda_{\mathrm{dim}(F)}\, d\mu\leq \chi_{\mathrm{min}}^F(x,f).\]

Let us now show the converse inequality.   For any $p$ there exist a subsequence $(n_{k,p})_k$ such that 
\[\beta_p(x)=\lim_k\frac{1}{n_{k,p}}\sum_{i=0}^{n_{k,p}-1} \log m\left(Df^p|_{F(f^ix)}\right).\]
Then if $\mu_p\in p\omega(x)$ is a weak limit of $\left(\mu^{n_{k,p}}_x\right)_k$ we have 
$$ \int\log m\left(Df^p|_{F(y)}\right)\, d\mu_p(y)= \beta_p(x).
$$

For any $f$-invariant measure $\mu$ and any $z\in U$, the sequences $
\left(\int  \log m\left(Df^p|_{F(y)}\right)\, d\mu(y)\right)_p$ and $(\beta_p(z))_p$ being 
both superadditive,  the terms $\chi^+(\mu)$ and $\chi^F_{\mathrm{min}}(z,f)$ are 
respectively the supremum of the sequences  $
\left(\frac{\int  \log m\left(Df^p|_{F(y)}\right)\, d\mu(y)}{p}\right)_p$ and $\left(\frac{\beta_p(z)}{p}\right)_p$.   We get therefore:
\begin{eqnarray*}\sup_{\mu\in p\omega(x)}\int \lambda_{\mathrm{dim}(F)}\, d\mu&=&\sup_{\mu\in p\omega(x)}\sup_p\frac{ \int \log 
m\left(Df^{p}|_{F(y)}\right)\, d\mu(y)}{p}.
\end{eqnarray*}
By   inverting the  two suprema  in the right-hand side term we conclude:
\begin{eqnarray*}\sup_{\mu\in p\omega(x)}\int \lambda_{\mathrm{dim}(F)}\, d\mu&=&
\sup_{p}\sup_{\mu\in p\omega(x)}\frac{ \int \log 
m\left(Df^{p}|_{F(y)}\right)\, d\mu(y)}{p},\\
&\geq &\sup_p \frac{ \int \log 
m\left(Df^{p}|_{F(y)}\right)\,  d\mu_{p}(y)}{p},\\
&\geq &\sup_p \frac{\beta_{p}(x)}{p}=\chi^F_{\mathrm{min}}(x,f).
\end{eqnarray*}

\end{proof}

\vskip 5pt

\begin{tabular}{l l l}

\emph{\normalsize David Burguet}

\medskip\\

\small LPSM
\\
\small Sorbonne Universit\' e
\\
\small Paris, France
\\
\texttt{david.burguet@upmc.fr}

\\
\\

\emph{\normalsize Dawei Yang}

\medskip\\

\small School of Mathematical Sciences
\\
\small Soochow University
\\
\small Suzhou, 215006, P.R. China
\\
\texttt{yangdw1981@gmail.com,yangdw@suda.edu.cn}

\end{tabular}

\end{document}